\documentclass[a4paper,11pt,reqno]{amsart}
\usepackage{amsmath}
\usepackage{geometry}
\usepackage[utf8]{inputenc}
\usepackage[T1]{fontenc}
\usepackage{latexsym,amsmath,amsfonts,amscd,amssymb,amsthm,mathrsfs}
\usepackage{pifont}
\usepackage{enumerate}
\usepackage[all]{xy}
\usepackage{color}
\usepackage{tikz}
\usepackage{lscape}
\usepackage{tabularx}
\usepackage{array}

\definecolor{cobalt}{RGB}{61,89,171}
\usepackage[colorlinks,citecolor=cobalt,linkcolor=cobalt,urlcolor=cobalt,pdfpagemode=UseNone]{hyperref}

\newcommand{\del}{\partial}

\newcommand{\Z}{\mathbb{Z}}
\newcommand{\R}{\mathbb{R}}

\newcommand{\C}{\mathbb{C}}
\newcommand{\Ker}{Ker}
\newcommand{\im}{Im}

\allowdisplaybreaks[1]

\theoremstyle{plain}
\newtheorem{theorem}{Theorem}[section]

\newtheorem{lemma}[theorem]{Lemma}

\theoremstyle{definition}
\newtheorem{definition}[theorem]{Definition}

\theoremstyle{remark}

\begin{document}
\title{Symplectic non-K\"ahler manifolds with and without the Hard Lefschetz Condition}
\author{Richard Hind}
\address[Richard Hind]{Department of Mathematics \\
University of Notre Dame\\
255 Hurley Bldg \\
Notre Dame, IN 46556}
\email{hind.1@nd.edu}

\author{Adriano Tomassini}
\address[Adriano Tomassini]{
Dipartimento di Scienze Matematiche, Fisiche e Informatiche\\
Unità di Mate\-matica e Informa\-tica\\
Università degli studi di Parma\\
Parco Area delle Scienze 53/A, 43124 Parma, Italy}
\email{adriano.tomassini@unipr.it}

\thanks{The first author is partially supported by the Simons Foundation. 
The second author is partially supported by the Project PRIN 2022 ``Real and Complex Manifolds: Geometry and Holomorphic Dynamics 2022AP8HZ9'' and by GNSAGA of INdAM}

\date{\today}
\begin{abstract}
In this paper we construct compact manifolds without K\"ahler structures that admit both a symplectic form satisfying the Hard Lefschetz Condition (HLC) and another symplectic form that does not. Our construction builds upon the orbifold introduced by Fern\'andez and Mu\~noz and its symplectic resolution studied by Cavalcanti, Fern\'andez, and Mu\~noz. By considering a one-parameter family of symplectic forms on the orbifold, we show that the corresponding resolved manifolds fail to satisfy the HLC for all parameters. However, after performing a suitable symplectic blowup along a union of tori, we obtain a family of symplectic manifolds for which the HLC holds for all non-zero parameters but fails at the central parameter. As a consequence, we exhibit a smooth manifold with no K\"ahler structure whose space of symplectic forms contains both HLC and non-HLC structures in the same connected component. This provides new examples of the subtle interplay between symplectic topology and the Hard Lefschetz property.
\end{abstract}
\maketitle

\section*{Introduction}
A compact $2n$-dimensional symplectic manifold $(X,\omega)$ is said to satisfy the \emph{Hard Lefschetz Condition} (HLC) if the maps
\begin{equation}\label{HLC}
[\omega]^k: H^{n-k}_{dR}(X) \longrightarrow H^{n+k}_{dR}(X)
\end{equation}
are isomorphisms for every $0 \leq k \leq n$. It is well known that if $(X,\omega,J)$ is a compact K\"ahler manifold, then $(X,\omega)$ satisfies the HLC (see e.g., \cite{GH}) and the de Rham complex $(\Omega^*(X), d)$ is a formal DGA in the sense of Sullivan (see \cite{DGMS}); consequently, any compact K\"ahler manifold provides an example of a symplectic manifold satisfying the HLC. Moreover, HLC symplectic manifolds share several cohomological properties with K\"ahler manifolds: the odd Betti numbers $b_{2k+1}(X)$ are even, the inequalities $b_k(X) \leq b_{k+2}(X)$ hold for $0 \leq k < n-1$, and the even Betti numbers satisfy $b_{2k}(X) > 0$, for $0\leq k\leq n$.

If $(X,\omega)$ is a $2n$-dimensional symplectic manifold, one can define a symplectic codifferential operator $d^\Lambda: \Omega^k(X) \to \Omega^{k-1}(X)$ using the symplectic star operator. A fundamental symplectic identity then holds:
\[
[d, \Lambda] = d^\Lambda,
\]
where $\Lambda$ denotes the symplectic adjoint of the Lefschetz operator $L(\alpha) = \omega \wedge \alpha$. In the symplectic setting, the triple $(\Omega^*(X), d, d^\Lambda)$ forms a \textit{differentiable Gerstenhaber–Batalin–Vilkovisky (dGBV) algebra}. This algebra is said to be \textit{integrable} (i.e., the $dd^\Lambda$-lemma holds) if and only if $(X,\omega)$ satisfies the HLC (see \cite{Yan, Cav2, Merkulov, Ma, Br, Mathieu}).
\vskip.1truecm\noindent 
On the other hand, families of compact symplectic manifolds with no K\"ahler structure that nonetheless satisfy the HLC can be obtained as compact quotients of connected, simply connected solvable Lie groups by lattices (see e.g., \cite{K2, K1, dBT1, AG1, LT}). In particular, Kasuya proved in \cite{K2} that if the action $\phi: \mathbb{R}^k \to \operatorname{Aut}(\mathbb{R}^m)$ is semisimple, then any symplectic solvmanifold $(\Gamma \backslash (\mathbb{R}^k \ltimes_\phi \mathbb{R}^m), \omega)$ satisfies the Hard Lefschetz condition for any symplectic form. Similar results appear in \cite{AG1}, \cite{LT} and \cite{AG2}. In these solvmanifold examples the HLC holds for \emph{every} symplectic form. 

The paragraphs above naturally raise the following question:
\vskip.1truecm
\begin{quote}
\emph{Do there exist compact $2n$-dimensional smooth manifolds that admit a symplectic form satisfying the HLC and another symplectic form that does not satisfy the HLC?}
\end{quote}
\vskip.2truecm\noindent
Cho, \cite{cho}, has constructed a simply connected compact K\"ahler manifold $(M,J,\omega)$ together with a symplectic form $\sigma$ on $M$ that does not satisfy the Hard Lefschetz Condition, even though $\sigma$ is symplectically deformation equivalent to the K\"ahler form $\omega$.

The aim of this paper is to provide compact manifolds with no K\"ahler structure that admit both a symplectic structure satisfying the HLC and a symplectic structure that does not. This is the non-K\"ahler counterpart of Cho's result: whereas the manifold in \cite{cho} is K\"ahler and the non-HLC form is a deformation of the K\"ahler form, the manifold $\hat{M}$ constructed here admits \emph{no} K\"ahler structure at all; in fact it is non-formal, carrying a nontrivial Massey product (see Section \ref{HLCblowup}). Still, its symplectic cone contains both HLC and non-HLC classes within a single connected component.

\bigskip

Given a symplectic orbifold $(M,\omega)$, Cavalcanti, Fern\'andez, and Mu\~noz, in \cite{CFM}, introduced the notion of a \emph{symplectic resolution} $(\tilde{M}, \tilde{\omega})$ of $(M,\omega)$. They showed by construction that symplectic resolutions always exist, \cite[Theorem 3.3]{CFM}, and further, that the resolutions $(\tilde{M}, \tilde{\omega})$ constructed satisfy the HLC if and only if $(M,\omega)$ does (see \cite[Theorem 3.9]{CFM}).

In \cite[Example 5.1]{CFM} a symplectic structure is defined on the orbifold $M$ constructed by Fern\'andez and Mu\~noz in \cite{FM}, and shown not to satisfy the HLC. Thus the resolution $(\tilde{M}, \tilde{\omega})$ also does \emph{not} satisfy the Hard Lefschetz Condition (see \cite[Example 5.2]{CFM}). However, by performing a symplectic blowup of $(\tilde{M}, \tilde{\omega})$ along the union of three suitably chosen symplectic tori $\bigcup_{i=1}^3 \mathbb{T}^2_i$, the article \cite{CFM} obtained a new symplectic manifold $(\hat{M}, \hat{\omega})$ that does satisfy the HLC. We emphasize that in \cite{CFM} the HLC and non-HLC phenomena are realized by two different manifolds: the resolution fails the HLC, while the blowup, equipped with its own symplectic form, satisfies it. By contrast, our construction produces a \emph{single} smooth manifold $\hat{M}$ carrying a \emph{one-parameter family} $\{\hat{\omega}_t\}$ of symplectic forms for $t \in (-\infty, 2)$, all lying in the same connected component of the symplectic cone, for which the HLC holds if and only if $t\neq 0$. It is this passage from ``two constructions, one form each'' to ``one manifold, a connected family of forms'' that yields Theorem \ref{intro2}.

We will define another symplectic structure $\omega'$ on the orbifold $M$ above, and prove the following results (see Theorems \ref{sec4main} and \ref{main}).
\begin{theorem} 
Let
$$
\omega_t=t\omega + (1-t)\omega',\qquad  t\in \R.
$$
Then $\{\omega_t\}_{ t\in \R}$ is a family of symplectic structures on the complex orbifold $M$ interpolating between $\omega'$ and $\omega$, and has the following properties:
\begin{enumerate}
    \item for all $t$, the symplectic resolution $(\tilde{M}, \tilde{\omega}_t)$ of $(M,\omega_t)$ does not satisfy the Hard Lefschetz Condition; \\
    \item for all $t\neq 0,2$, the symplectic blowup $(\hat{M},\hat{\omega}_t)$ of $\tilde{M}$ along the union of three suitably chosen symplectic tori $S=\bigcup_{i=1}^3 \mathbb{T}^2_i$ satisfies the Hard Lefschetz Condition; setting $t=0$, the symplectic blowup $(\hat{M},\hat{\omega'})$ along $S$ 
    exists, but
    does not satisfy the Hard Lefschetz Condition.
\end{enumerate}
\end{theorem}
In the first statement here we can use any of the resolutions constructed by \cite[Theorem 3.3]{CFM}.
For the second statement, the symplectic blowups are equipped with symplectic forms such that the parameter $\epsilon$ defined in Section \ref{symplectic-blowup} is small.

As a consequence we have, see Theorem \ref{main}, 
\begin{theorem}\label{intro2}
The smooth manifold $\hat{M}$ has no K\"ahler structures, but the space of symplectic forms is nonempty and admits both HLC and non-HLC structures in the same connected component. 
\end{theorem}

We note that the Hard Lefschetz Condition \eqref{HLC} depends only on the cohomology class of the symplectic form, and indeed makes sense for all cohomology classes $\alpha \in H^2_{dR}(X)$. Our final remark is that if a manifold $X$ admits at least one class $\alpha$ satisfying the HLC, then the difficulty in statements like Theorem \ref{intro2} lies in finding a symplectic form which does not satisfy the HLC. Indeed, Lemma \ref{lem:openHLC} below shows that the HLC classes are open and dense once a single one exists, so the genuine task is to exhibit a non-HLC form within the connected component of an HLC form, which is exactly what the central value $t=0$ achieves. We can make this precise as follows.

\begin{lemma}\label{lem:openHLC}
    Suppose $X^{2n}$ is a smooth manifold and $\alpha \in H^2_{dR}(X)$ satisfies the HLC. Then the set of classes $\beta \in H^2_{dR}(X)$ which satisfy the HLC is both open and dense.
\end{lemma}

\begin{proof}
    The open condition is clear.To establish denseness, suppose $\beta \in H^2_{dR}(X)$ and consider the 1-parameter family of classes $\alpha_t = (1-t) \beta + t \alpha$. Fixing $k$ and choosing bases for $H^{n-k}_{dR}(X)$ and $H^{n+k}_{dR}(X)$, the map $\alpha^k_t:H^{n-k}_{dR}(X) \to H^{n+k}_{dR}(X)$ is represented by a matrix whose determinant $\det(\alpha_t^k)$ is a polynomial in $t$. By assumption $\det(\alpha_1^k) = \det(\alpha^k) \neq 0$, so the polynomial is nonzero, and in particular $\alpha_t^k$ is an isomorphism for $t \neq 0$ sufficiently small.
\end{proof}

\vskip.2truecm

The paper is organized as follows. In Section~\ref{preliminaries}, we recall some basic facts about orbifolds. Section~\ref{symplectic-blowup} reviews the construction of the symplectic blowup due to McDuff \cite{McD} and recalls the definition of the symplectic resolution of a symplectic orbifold given in \cite{CFM}. Section~\ref{HLCblowup} is devoted to the proofs of Theorems~\ref{sec4main} and \ref{main}. 
\vskip.2truecm
\noindent{\sl Acknowledgments. The authors would like to thank Gil Cavalcanti for his interest in the paper and useful discussions. The second author would like to thank Yasha Eliashberg for the interesting, fruitful, and stimulating discussions on the subject during his visit to Stanford. He would also like to thank the Stanford Department of Mathematics for its warm hospitality.}

\section{Preliminaries}\label{preliminaries}

We start by recalling the definition of a complex orbifold, as originally introduced by \cite{SA56} (see also \cite{Bai56} and \cite{Joy00}).
\begin{definition}
An {\em orbifold} is a singular real manifold $M$ of dimension $n$ whose singularities are locally isomorphic to quotient singularities $\R^n/ G$, where $G$ is a finite subgroup of $GL(n;\R)$, such that if $1\neq \gamma\in G$, then the subspace $V_\gamma$ of $\R^n$ fixed by $\gamma$ has $\dim_\R V_\gamma\leq n-2$.
\end{definition}
Similarly, 
\begin{definition}
A \emph{complex orbifold of complex dimension} $n$ is a singular complex manifold $M$ of dimension $n$ whose singularities are locally isomorphic to quotient singularities $\C^n/ G$, where $G$ is a finite subgroup of $GL(n;\C)$.
\end{definition}
In Section \ref{HLCblowup} we will consider the following class of complex orbifolds.
\begin{definition}
A complex orbifold $M$ is said to be of \emph{global-quotient-type} if $M=X/ G$, where $X$ is a complex manifold and $G$ is a finite subgroup of the group of biholomorphisms of $X$.
\end{definition}
We will denote by $\Delta$ the set of singular points of $M$. The complement $M\setminus \Delta$ is a smooth, respectively complex, manifold.\newline
Tensors on real, respectively complex, orbifolds $M$, such as vector fields, differential forms, or metrics, are  defined locally at $x\in M$ as $G_x$-invariant tensors on $\R^n$, respectively $\C^n$, where $G_x\subset GL(n;\R)$, respectively $G_x\subset GL(n;\C)$, is such that $M$ is locally isomorphic to $\R^n/G_x$, respectively $\C^n/G_x$, at $x$.

Given a complex orbifold $M$, let $(\mathcal{A}_{\C}^{\bullet}(M),d)$ be the space of orbifold complex forms; then the orbifold de Rham cohomology is defined as  
\begin{equation}\label{eq:coom_orb}
\textstyle H_{dR}^{k}(M;\C)=\frac{\Ker\, d}{\im\, d}\cap \mathcal{A}^{k}_\C(M)
\end{equation}

Once we fix a Hermitian metric $g$ on a compact complex orbifold $M$ of complex dimension $n$, 
one can define the de Rham Laplacian $\Delta$ and its kernel $$
\mathcal{H}_{d}^k(M,g)=\{\alpha\in\textstyle\mathcal{A}^k(M):\Delta\alpha=0\}.
$$ 


It turns out that an orbifold $k$-form on a compact orbifold $M$ is harmonic if and only if it is $d$-closed and $d^{\ast}$-closed, 
and the following isomorphism holds
\begin{align*}
H_{dR}^k(M;\C)&\rightarrow\mathcal{H}_{d}^k(M,g),
\end{align*}
that is, the $k$th de Rham orbifold cohomology group is computed by the $G$-invariant harmonic $k$-forms on $M$.

\section{Symplectic blowup and resolution of symplectic orbifolds}\label{symplectic-blowup}
We start by reviewing the construction of symplectic blowup by McDuff \cite{McD}. Let $(X,\omega)$ be a $2n$-dimensional symplectic manifold and let $\iota:(Y,\sigma)\hookrightarrow (X,\omega)$ be a symplectic embedding, where $Y$ is a compact $2m$-dimensional symplectic submanifold of $X$, and set $k=n-m$. Then the structure group of the normal bundle $\mathcal{N}_Y\vert_X$ in $X$ reduces to $U(k)$. By identifying $Y$ with the image $\iota(Y)\subset X$, we can view $\pi:\mathcal{N}_Y\vert_X\to Y$ as the symplectic orthogonal bundle structure of the tangent bundle $TY$ in $TX$; hence $\mathcal{N}_Y\vert_X$ has a canonical symplectic bundle structure given by the restriction of $\omega$ to each fibre, and hence a homotopically unique compatible complex structure. With respect to this complex structure, we consider the projectivisation $\mathbb{P}(\mathcal{N}_Y\vert_X)$ of $\mathcal{N}_Y\vert_X$. Choose a tubular neighbourhood  $W$ of $Y$ in $X$. Then there exists a closed $2$-form $\rho$ on the total space of $\mathcal{N}_Y\vert_X$ which restricts to $\sigma$ along the zero section and to the canonical symplectic form on each fibre, and with respect to which $W$ may be symplectically identified with a neighbourhood $V$ of the zero section of $\mathcal{N}_Y\vert_X$ \cite[Lemma 3.1]{McD}. Let $$\mathcal{L}=\{(\ell,\nu)\in \mathbb{P}(\mathcal{N}_Y\vert_X)\times\mathcal{N}_Y\vert_X\,\,\vert\,\,\nu\in\ell\}$$
be the tautological bundle over $\mathbb{P}(\mathcal{N}_Y\vert_X)$. Then $\mathcal{L}$ fibres over $Y$ and each fibre is a line bundle over $\mathbb{P}^{k-1}$. Let $q:\mathcal{L}\to\mathbb{P}(\mathcal{N}_Y\vert_X)$ and $\varphi:\mathcal{L}\to \mathcal{N}_Y\vert_X$ be the natural projections, then we have the following commutative diagram
$$
\xymatrix{
\mathcal{L} \ar[r]^{\!\!\!\!\!\!\!q} \ar[d]^\varphi & \mathbb{P}(\mathcal{N}_Y\vert_X) \ar[d]^p\\
\mathcal{N}_Y\vert_X \ar[r]^\pi & Y .
}
$$
The projection $\varphi$ is a diffeomorphism outside the zero section of $\mathcal{L}$.
Let $\tilde{V}=\varphi^{-1}(V)$. Define 
$$
\tilde{X}=\overline{X\setminus W}\cup_{\del\tilde{V}}\tilde{V};
$$
then the map $\varphi$ can be extended to a map $f:\tilde{X}\to X$, being the identity in the complement of $\tilde{V}$. Then, $\tilde{X}$ is the {\em blowup} of $X$. By \cite[Proposition 2.4]{McD}, there is a short exact sequence 
$$ 
0 \to H^*_{dR}(X;\R) \to H^*_{dR}(\tilde{X};\R){\to} A^* \to 0, 
$$
where $A^\ast$ is a free module over $H^*_{dR}(Y;\R)$ with generators $a,\ldots,a^{k-1}$, where $a$ is the first Chern class of the dual of the tautological line bundle over $\mathcal{L}$. Furthermore, by \cite[Proposition 3.7]{McD}, there is a representative $\alpha$ of $a$ with compact support in $\tilde{V}$ such that, for $\epsilon$ small enough, 
$$
\tilde{\omega}=f^*\omega + \epsilon\alpha
$$
is a symplectic structure on $\tilde{X}$. \newline 
As already remarked in \cite[Remark p. 336]{Cav1}, even up to diffeomorphism the symplectic structure in the blowup $\tilde{X}$ is not determined by the symplectic structure $\omega$ on $X$. Indeed, the cohomology class depends on $\epsilon$ and (see \cite[page 231]{McDS}), for a given $\epsilon$, the symplectic structure on $\tilde{X}$ in principle depends on the almost complex structure compatible with $\omega$ and on the identification of the neighbourhood $V$ of the zero section of the normal bundle $\mathcal{N}_Y\vert_X$ with the tubular neighbourhood $W$.\vskip.2truecm
We recall now the definition of a \emph{symplectic resolution} (see \cite[Definition 3.2]{CFM}). Let $(M,\omega)$ be a symplectic orbifold and let $\Delta$ be the set of singular points of $M$. A \emph{symplectic resolution} of $(M,\omega)$ is a smooth symplectic manifold $(\tilde M,\tilde \omega)$ together with a map $\pi:\tilde M\to M$, such that setting $E=\pi^{-1}(\Delta)$, called the \emph{exceptional set}, the following hold:
\begin{enumerate}
    \item $\pi: \tilde M\setminus E \to M\setminus \Delta$ is a diffeomorphism;\\[5pt]
    \item $E$ is a union of possibly intersecting smooth symplectic submanifolds of $\tilde M$ of codimension at least $2$;\\[5pt]
    \item $\tilde \omega = \pi^*\omega$ on the complement of a neighbourhood of the exceptional set $E$.
\end{enumerate}
Then, in \cite[Theorem 3.3]{CFM} it is proved that any symplectic orbifold has a symplectic resolution.

\section{Hard Lefschetz Condition under blowups}\label{HLCblowup}
We start by reviewing the orbifold construction by Fern\'{a}ndez and Mu\~{n}oz \cite{FM}. Let $\mathbb{H}(3;\C)$ be the complex $3$-dimensional Heisenberg group and let $G=\mathbb{H}(3;\C)\times \C$, identified with the complex matrices of the following form
$$G=\left\{\left(
\begin{array}{lllll}
1 & z_1 & z_3 & 0 & 0\\
0 &  1& z_2 & 0 & 0\\
0 & 0 & 1 & 0 & 0\\
0 & 0& 0 & 1 & z_4\\
0 & 0& 0 & 0 & 1
\end{array}
\right)\,\,\,\Big\vert\,\,\, z_1,z_2,z_3,z_4\in\C \right\}.
$$
Let $\xi$ be a primitive cubic root of unity and set 
$$
\Lambda=\{a+b\xi\,\,\,\vert\,\,\, a, b\in\Z\}
$$
and let $\Gamma\subset G$ be formed by matrices of $G$ having entries in $\Lambda$. Then $\Gamma$ is a lattice in $G$ and $\mathbb{I}_4=\Gamma\backslash G$ is a $4$-dimensional compact complex manifold. Let $\rho:G\to G$  be defined by
$$
\rho(z)=(\xi z_1,\xi z_2, \xi^2z_3, \xi z_4).
$$ 
Then $\rho$ induces a biholomorphism of $\mathbb{I}_4$ and $\rho^3=id$; one can check that there are $81$ fixed points so that $M=\mathbb{I}_4\slash \langle \rho\rangle$ is a compact complex orbifold of global-quotient-type. Let $\{\eta^1,\ldots,\eta^4\}$ be the $(1,0)$-coframe on $\mathbb{I}_4$ defined as
$$
\eta^1=dz_1,\quad \eta^2=dz_2,\qquad \eta^3=dz_3-z_1dz_2,\qquad \eta^4=dz_4;
$$
then 
$$
\rho^*\eta^1=\xi\eta^1,\qquad \rho^*\eta^2=\xi\eta^2,\qquad \rho^*\eta^3=\xi^2\eta^3,\qquad \rho^*\eta^4=\xi\eta^4.
$$
Accordingly,
\begin{equation}
\omega =\frac{i}{2}\eta^{1\overline{1}}+\frac{1}{2}\eta^{23}+\frac{1}{2}\eta^{\overline{2}\overline{3}} +\frac{i}{2}\eta^{4\overline{4}}
\end{equation}
is a symplectic structure on the orbifold $M$ (see \cite[p.596]{CFM}). Taking real and imaginary parts, we write 
$$
\eta^j=e^{2j-1}+ie^{2j}\quad, 1\leq j\leq 4,
$$
then $\omega=e^{12}+e^{35}-e^{46}+e^{78}$.

Let $\mathbb{T}^2_1$, $\mathbb{T}^2_2$ and $\mathbb{T}^2_3$ be three $2$-tori which integrate the 
three involutive distributions given respectively by 
\begin{equation}\label{2-tori}
\mathcal{D}_1=\R\langle e_3+e_7,e_4+e_8\rangle, \,\,\,\mathcal{D}_2=\R\langle e_3+\sqrt{3}e_8,e_7\rangle, \,\,\,\mathcal{D}_3=\R\langle e_3+e_7,e_8\rangle
\end{equation}
where $\{e_1,\ldots,e_8\}$ is the dual frame of $\{e^1,\ldots,e^8\}$. Then $\mathbb{T}^2_1$, $\mathbb{T}^2_2$ and $\mathbb{T}^2_3$ are symplectic with respect to $\omega$, and following
Cavalcanti, Fern\'{a}ndez and Mu\~{n}oz \cite[p.598]{CFM}, we may assume they are disjoint from both each other and from the fixed points of $\rho$.

By \cite[Theorem 3.3]{CFM}, a symplectic orbifold $(N, \sigma)$ has a symplectic resolution $(\tilde{N}, \tilde{\sigma})$ (in the sense of \cite[Definition 3.2]{CFM}, see section \ref{symplectic-blowup}), and by \cite[Theorem 3.9]{CFM} the symplectic manifold $(\tilde{N}, \tilde{\sigma})$ constructed satisfies the Hard Lefschetz Condition if and only if the orbifold $(N, \sigma)$ does.

In the case of our $(M, \omega)$, \cite[Example 5.2]{CFM} shows that the Hard Lefschetz Condition does not hold, and therefore the constructed resolution $(\tilde{M}, \tilde{\omega})$ is a symplectic manifold which does not satisfy the Hard Lefschetz Condition. However, we can blowup $(\tilde{M}, \tilde{\omega})$ along the union of the three symplectic tori $S=\bigcup_{i=1}^3 \mathbb{T}^2_i$ to get a new symplectic manifold which we denote by $(\hat{M}, \hat{\omega})$. The formula for the kernel of $[\hat{\omega}]^2$ in \cite[Theorem 4.1]{CFM} shows that, when the blowup is defined as in Section \ref{symplectic-blowup}, $(\hat{M}, \hat{\omega})$ does satisfy the Hard Lefschetz Condition.


On the manifold $\mathbb{I}_4$ we now consider the alternative form $\omega'$ defined by
\begin{equation}
\omega'=\frac{1}{2}\eta^{13}+\frac{1}{2}\eta^{\overline{1}\overline{3}}+\frac{i}{2}\eta^{2\overline{2}} +\frac{i}{2}\eta^{4\overline{4}}=e^{15}-e^{26}+e^{34}+e^{78}
\end{equation}
Then $\omega'$ is a real, closed, nondegenerate and $\rho$-invariant $2$-form on $\mathbb{I}_4$, hence it descends to $M$ such that $(M,\omega')$ is a symplectic orbifold. The tori $\mathbb{T}^2_i$ are also symplectic with respect to $\omega'$. 

More generally, define
\[
\omega_t = t \omega + (1-t) \omega'.
\]
A direct computation gives
\begin{align*}
\omega_t(e_3+e_7,e_4+e_8) &= 2-t,\\
\omega_t(e_3+\sqrt{3}\,e_8,e_7) &= -\sqrt{3},\\
\omega_t(e_3+e_7,e_8) &= 1.
\end{align*}
Hence the torus $\mathbb{T}^2_1$ is symplectic with respect to $\omega_t$ $\forall t\in\R, t\neq 2$, $\mathbb{T}^2_2$ and $\mathbb{T}^2_3$ are symplectic $\forall t\in\R$. Therefore the symplectic blowup of $\tilde{M}$ along $S=\bigcup_{i=1}^3 \mathbb{T}^2_i$ is defined $\forall t\in\R,t\neq 2$.

We are ready to prove the following
\begin{theorem}\label{sec4main} 
Let
$$
\omega_t=t\omega + (1-t)\omega',\qquad  t\in \R
$$
Then $\{\omega_t\}_{ t\in \R}$ is a family of symplectic structures on the complex orbifold $M$ such that 
\begin{enumerate}
    \item $$\omega_0=\omega',\qquad \omega_1=\omega ;$$
    and $\omega_t$ induces the same orientation as $$
\Omega:=\frac{1}{4!}\Big(\frac{i}{2}\sum_{r=1}^4\eta^{r\overline{r}}\Big)^4.
$$\\
    \item for all $t$, the symplectic resolution $(\tilde{M}, \tilde{\omega}_t)$ of $(M,\omega_t)$ does not satisfy the Hard Lefschetz Condition;\\[2pt]
    \item for all $t\neq 0,2$, the symplectic blowup $(\hat{M},\hat{\omega}_t)$ of $(\tilde{M},\tilde{\omega}_t)$ along the union of the three symplectic tori $S=\bigcup_{i=1}^3 \mathbb{T}^2_i$ satisfies the Hard Lefschetz Condition, but setting $t=0$, the symplectic blowup $(\hat{M},\hat{\omega'})$ along $S$ 
    does not satisfy the Hard Lefschetz Condition.
\end{enumerate}
\end{theorem}

Regarding point (2), in principle the resolutions $\tilde{M}$ constructed in \cite[Theorem 3.3]{CFM} depend on the symplectic form $\omega_t$ on $M$. However, in this case, since local charts around singular points in $M$ are unchanged up to equivariant symplectomorphism, we may identify these resolutions with a fixed underlying smooth manifold $\tilde{M}$, and then the $\tilde{\omega}_t$ are a smooth family of forms on $\tilde{M}$.
Point (3), as usual, relies on taking blowups with small parameter $\epsilon$.

\begin{proof} (1) First we prove that $\omega_t$ is nondegenerate. We have the following:
$$
\omega_t=t\omega + (1-t)\omega'=t(e^{12}+e^{35}-e^{46}+e^{78})+ (1-t)(e^{15}-e^{26}+e^{34}+e^{78}),
$$
for $t\in\R$. A straightforward computation yields
$$
\omega_t^4=24\Big(3t^2-3t+1\Big)e^{12345678};
$$
since
$$3t^2-3t+1>0\,\qquad \forall t\in\R,$$
it follows that $\omega_t$ is nondegenerate and induces the same orientation as 
$$
\Omega=\frac{1}{4!}\Big(\frac{i}{2}\sum_{r=1}^4\eta^{r\overline{r}}\Big)^4=e^{12345678}.
$$
Clearly, $d\omega_t=0$ for every $t\in\R$ and consequently $\omega_t$ is a family of symplectic structures on the orbifold $M$ with 
$$\omega_{t=0}=\omega',\,\qquad \omega_{t=1}=\omega.$$

(2) and (3) We start by dealing with the case $t=0$. First we show that $\tilde{\omega}_0= \tilde{\omega'}$ does not satisfy the Hard Lefschetz Condition. By \cite[Theorem 3.9]{CFM}, it suffices to work on $M$, and we will study the action of $[\omega']^2$ on $H^2_{dR}(M)$.


The $2$-nd de Rham orbifold complex cohomology group $H^2_{dR}(M;\C)$ is given by the $\rho$-invariant 
$2$-cohomology classes in $H^2_{dR}(\mathbb{I}_4;\C)$.  Explicitly, we have
\begin{equation}\label{2-cohomology-basis}
H^2_{dR}(M;\C)=\C\langle \eta^{1\overline{1}}, \eta^{2\overline{2}}, \eta^{4\overline{4}}, \eta^{1\overline{2}}, \eta^{2\overline{1}}, \eta^{1\overline{4}}, \eta^{4\overline{1}}, 
\eta^{2\overline{4}}, \eta^{4\overline{2}}, \eta^{13},  \eta^{23},  \eta^{\overline{1}\overline{3}}, \eta^{\overline{2}\overline{3}}\rangle,
\end{equation}
that is $H^2_{dR}(M;\C)$ is a $13$-dimensional complex vector space.\newline 
Given this, a
direct calculation yields 
\begin{equation}\label{ker-formula}
\ker([\omega']^2\cup)=\R\langle \frac{i}{2} [\eta^{1\overline{1}}],  \frac12[\eta^{1\overline{4}}+ \eta^{\overline{1}4}],  \frac{i}{2}[\eta^{\overline{1}4}-\eta^{1\overline{4}} ]\rangle
=\R\langle [e^{12}],  [e^{17}+e^{28}],  [e^{27}-e^{18}]\rangle,
\end{equation}
in particular it is nontrivial as claimed.

Moving to the blowup along the symplectic tori, according to \cite{Cav1} and \cite[Theorem 4.1]{CFM},  the kernel of $[\hat{\omega'}]^2$ can be computed via the following 
\begin{equation}\label{orbifold-formula}
\ker([\hat{\omega'}]^2\cup)\!=\!\pi^*\Big(\ker([\omega']^2\cup)\cap\ker(\hbox{PD}(\mathbb{T}^2_1)\cup)\cap\ker(\hbox{PD}(\mathbb{T}^2_2)\cup)\cap\ker(\hbox{PD}(\mathbb{T}^2_3)\cup)\Big)
\end{equation}
where $\hbox{PD}$ denotes the Poincar\'e dual and $\pi: \hat{M} \to M$ is the composition of the symplectic blowup and the symplectic resolution. The formula assumes we blowup with the parameter $\epsilon$ in section \ref{symplectic-blowup} sufficiently small. (For certain choices of $\epsilon$ the subspace $\ker([\omega']^2\cup)$ may be larger.)

For the tori $\mathbb{T}^2_j$, $j=1,2,3$, a straightforward computation yields
\begin{equation}\label{poincare-dual-tori}
\left\{
\begin{array}{l}
\hbox{PD}(\mathbb{T}^2_1)=e^{125678}+e^{124567}-e^{123568}+e^{123456};\\[10pt]
\hbox{PD}(\mathbb{T}^2_2)=-e^{124568}-\sqrt{3}e^{123456};\\[10pt]
\hbox{PD}(\mathbb{T}^2_3)=e^{124567}+e^{123456}.
\end{array}
\right.
\end{equation}

By \eqref{ker-formula}, the kernel of $[\omega']^2$ pairs trivially with each of the Poincar\'{e} dual classes in \eqref{poincare-dual-tori}, and so by
formula \eqref{orbifold-formula} we see that 
$$
\ker([\hat{\omega'}]^2\cup)= \pi^*\Big(\ker([\omega']^2\cup) \Big) = \pi^*\Big(\R\langle [e^{12}],  [e^{17}+e^{28}],  [e^{27}-e^{18}]\rangle\Big).
$$
Hence, $[\hat{\omega'}]^2\cup: H^2_{dR}(\hat{M},\R)\to H^6_{dR}(\hat{M},\R)$ has a non trivial kernel, in other words, the blowup $(\hat{M}, \hat{\omega'})$ of $(\tilde{M},\tilde{\omega'})$ does not satisfy the HLC.\vskip.2truecm\noindent

To complete the proof of Theorem \ref{sec4main} we will show that, for all $t\in\R$, the symplectic resolution $(\tilde{M},\tilde{\omega}_t)$ of $(M,\omega_t)$ does not satisfy HLC; however, for $t \neq 0, 2$ its blowup does satisfy HLC. As already recalled, in the case $t=1$ this is a result of Cavalcanti, Fern\'{a}ndez and Mu\~{n}oz \cite[p.598]{CFM}. 
We also recall that for $t=2$ the first torus $\mathbb{T}_1^2$ is not symplectic.
By expanding $\omega^2_t$, we obtain 
$$
\begin{array}{lll}
\omega_t^2&=&2t^2\Big( i\eta^{123\overline{1}}+
i\eta^{1\overline{1}\overline{2}\overline{3}}
+\eta^{14\overline{1}\overline{4}}
+\eta^{23\overline{2}\overline{3}}
+i\eta^{234\overline{4}}
+i\eta^{4\overline{2}\overline{3}\overline{4}} \Big) \\
&+&2(1-t)^2\Big( \eta^{13\overline{1}\overline{3}}-
i\eta^{123\overline{2}}
+i\eta^{134\overline{4}}
-i\eta^{2\overline{1}\overline{2}\overline{3}}
+i\eta^{4\overline{1}\overline{3}\overline{4}}
+\eta^{24\overline{2}\overline{4}}
\Big)\\
&+&2t(1-t)\Big( \eta^{12\overline{1}\overline{2}}+
\eta^{14\overline{1}\overline{4}}
+\eta^{23\overline{1}\overline{3}}+
\eta^{13\overline{2}\overline{3}}
+i\eta^{234\overline{4}}
+i\eta^{4\overline{2}\overline{3}\overline{4}}
+i\eta^{134\overline{4}}
\\
&+&i\eta^{4\overline{1}\overline{3}\overline{4}}+\eta^{24\overline{2}\overline{4}}\Big).
\end{array}
$$
Consider the natural pairing
$$
H^2_{dR}(M;\C)\times H^2_{dR}(M;\C)\to \C
$$
defined as 
$$
([\alpha],[\beta])\mapsto\int_{M} \omega^2_t\wedge\alpha\wedge\beta.
$$
With respect to the basis given by \eqref{2-cohomology-basis}, up to a non-zero constant, the pairing is described by the following table. \vskip.5truecm

\begingroup  
\scriptsize
\setlength{\tabcolsep}{1pt}
\begin{tabular}{|c |ccccccccccccc|}
\hline 
\rule{0pt}{4ex} 
& $\eta^{1\bar{1}}$ & $\eta^{2\bar{2}}$ & $\eta^{4\bar{4}}$ & $\eta^{23}$ & $\eta^{\bar{2}\bar{3}}$ & $\eta^{\bar{1}2}$ & $\eta^{13}$ & $\eta^{1\bar{2}}$ & $\eta^{\bar{1}\bar{3}}$ & $\eta^{1\bar{4}}$ & $\eta^{\bar{1}4}$ & $\eta^{2\bar{4}}$ & $\eta^{\bar{2}4}$ \\[3pt]
\hline
 $\eta^{1\bar{1}}$ & $0$ & $0$ & $-2t^2$ & $-2it$ & $-2it$ & $0$ & $0$ & $0$ & $0$ & $0$ & $0$ & $0$ & $0$ \\[2pt]
$\eta^{2\bar{2}}$ & $0$ & $0$ & $-2(1-t)^2$ & $0$ & $0$ & $0$ & $-2i(1-t)$ & $0$ & $-2i(1-t)$ & $0$ & $0$ & $0$ & $0$ \\[2pt]
$\eta^{4\bar{4}}$ & $-2t^2$ & $-2(1-t)^2$ & $0$ & $-2it^2$ & $-2it^2$ & $-2t(1-t)$ & $-2i(1-t)^2$ & $2t(1-t)$ & $-2i(1-t)^2$ & $0$ & $0$ & $0$ & $0$ \\[2pt]
$\eta^{23}$ & $-2it$ & $0$ & $-2it^2$ & $0$ & $2t$ & $0$ & $0$ & $2i(1-t)$ & $0$ & $0$ & $0$ & $0$ & $0$ \\[2pt]
$\eta^{\bar{2}\bar{3}}$ & $-2it$ & $0$ & $-2it^2$ & $2t$ & $0$ & $-2i(1-t)$ & $0$ & $0$ & $0$ & $0$ & $0$ & $0$ & $0$ \\[2pt]
$\eta^{\bar{1}2}$ & $0$ & $0$ & $-2t(1-t)$ & $0$ & $-2i(1-t)$ & $0$ & $-2it$ & $0$ & $0$ & $0$ & $0$ & $0$ & $0$ \\[2pt]
$\eta^{13}$ & $0$ & $-2i(1-t)$ & $-2i(1-t)^2$ & $0$ & $0$ & $-2it$ & $0$ & $0$ & $2(1-t)$ & $0$ & $0$ & $0$ & $0$ \\[2pt]
$\eta^{1\bar{2}}$ & $0$ & $0$ & $2t(1-t)$ & $2i(1-t)$ & $0$ & $0$ & $0$ & $0$ & $2it$ & $0$ & $0$ & $0$ & $0$ \\[2pt]
$\eta^{\bar{1}\bar{3}}$ & $0$ & $-2i(1-t)$ & $-2i(1-t)^2$ & $0$ & $0$ & $0$ & $2(1-t)$ & $2it$ & $0$ & $0$ & $0$ & $0$ & $0$ \\[2pt]
$\eta^{1\bar{4}}$ & $0$ & $0$ & $0$ & $0$ & $0$ & $0$ & $0$ & $0$ & $0$ & $0$ & $-2t^2$ & $0$ & $2t(1-t)$ \\[2pt]
$\eta^{\bar{1}4}$ & $0$ & $0$ & $0$ & $0$ & $0$ & $0$ & $0$ & $0$ & $0$ & $-2t^2$ & $0$ & $2t(1-t)$ & $0$ \\[2pt]
$\eta^{2\bar{4}}$ & $0$ & $0$ & $0$ & $0$ & $0$ & $0$ & $0$ & $0$ & $0$ & $0$ & $2t(1-t)$ & $0$ & $-2(1-t)^2$ \\[2pt]
$\eta^{\bar{2}4}$ & $0$ & $0$ & $0$ & $0$ & $0$ & $0$ & $0$ & $0$ & $0$ & $2t(1-t)$ & $0$ & $-2(1-t)^2$ & $0$ \\[2pt]
\hline
\end{tabular}
\endgroup  
\vskip.5truecm\noindent
Denoting by $P_t\in M_{13,13}(\C)$ the matrix given by the table above, a straightforward computation shows that
\begin{equation}
\label{rkP2}
\hbox{rank}\,P_t=10,\,\,\hbox{\rm for every}\, t\in\R.
\end{equation}
Note that the value $t = 2$ plays no role at this stage, since the tori enter only in the blowup; the rank of $P_t$ is $10$ for all $t \in\R$.
Therefore, using \cite[Proposition 3.8, Theorem 3.9]{CFM}, we have that 
\begin{equation}\label{kertilde}
\ker\!\left([\tilde\omega_t]^{\,4-k}\colon H^{k}_{dR}(\tilde M)\to H^{8-k}_{dR}(\tilde M)\right)
\cong
\ker\!\left([\omega_t]^{\,4-k}\colon H^{k}_{dR}(M)\to H^{8-k}_{dR}(M)\right),\quad 1\leq k\leq 4.
\end{equation}
and consequently by \eqref{rkP2} and \eqref{kertilde} we obtain
$$
\dim_\C\ker\Big([\tilde{\omega}_t]^2:H^2_{dR}(\tilde{M};\C)\to H^6_{dR}(\tilde{M};\C)\Big)=3\,\,\,\,\hbox{\rm for every}\, t\in\R.
$$
Therefore, for all $t$, the symplectic resolution $(\tilde{M}, \tilde{\omega}_t)$ of $(M,\omega_t)$ does not satisfy the Hard Lefschetz Condition, and (2) is proved.
\vskip.4truecm
Now we prove (3) in the cases when $t \neq 0, 2$. First of all note that we need only to check that 
$$[\hat\omega_t]^{2}\colon H^{2}_{dR}(\hat M)\to H^{6}_{dR}(\hat M)
$$
is an isomorphism. Indeed, for the remaining levels, as already recalled $b_1(M)=b_3(M)=0$, so that 
$$H^{1}_{dR}(M)=H^{3}_{dR}(M)=0,$$
and hence these kernels vanish for $k=1,3$. Then, as the tori $\mathbb{T}^2_i$ have real dimension
$2$, Theorem~2.1 of \cite{Cav1} applies: at level $k=3$ its first item gives
\[
\ker\!\left([\hat\omega_t]\colon H^{3}_{dR}(\hat M)\to H^{5}_{dR}(\hat M)\right)
\cong
\ker\!\left([\tilde\omega_t]\colon H^{3}_{dR}(\tilde M)\to H^{5}_{dR}(\tilde M)\right)
\cong
\ker\!\left([\omega_t]\colon H^{3}_{dR}(M)\to H^{5}_{dR}(M)\right)=0,
\]
while at level $k=1$ its third item gives
\[
\begin{array}{ll}
\dim\ker\!\left([\hat\omega_t]^{3}\colon H^{1}_{dR}(\hat M)\to H^{7}_{dR}(\hat M)\right)
&\le
\dim\ker\!\left([\tilde\omega_t]^{3}\colon H^{1}_{dR}(\tilde M)\to H^{7}_{dR}(\tilde M)\right)\\[10pt]
&=\dim\ker\!\left([\omega_t]^{3}\colon H^{1}_{dR}( M)\to H^{7}_{dR}( M)\right)=0.
\end{array}
\]
A direct calculation, valid for every $t\in\R$, yields 
\begin{equation}\label{eq:ker-fixed}
\ker\!\big([\omega_t]^2\cup\big)=\mathbb{R}\big\langle a_1,\ a_2,\ a_3\big\rangle,
\end{equation}
where
\begin{eqnarray*}
a_1&=&[(1-t)^2e^{12}+t^2e^{34}+t(1-t)(e^{14}-e^{23})]=\frac{i}{2}[(1-t)^2\eta^{1\bar{1}}+t^2\eta^{2\bar{2}}+t(1-t)(\eta^{1\bar{2}}-\eta^{\bar{1}2})]\\
a_2&=&[(1-t)(e^{17}+e^{28})+t(e^{37}+e^{48})]=\frac{1}{2}[(1-t)(\eta^{1\bar{4}}+\eta^{\bar{1}4})+t(\eta^{2\bar{4}}+\eta^{\bar{2}4})]\\
a_3&=&[(1-t)(e^{27}-e^{18})+t(e^{47}-e^{38})]=\frac{i}{2}[(1-t)(\eta^{\bar{1}4}-\eta^{1\bar{4}})+t(\eta^{\bar{2}4}-\eta^{2\bar{4}})].
\end{eqnarray*}
%
Let
\[
[\alpha]=c_1\,a_1+c_2\,a_2+c_3\,a_3\in\ker\!\big([\omega_t]^2\cup\big).
\]
Direct computation using \eqref{poincare-dual-tori} gives
\begin{equation}\label{eq:PD-fixed}
\left\{
\begin{aligned}
\hbox{PD}(\mathbb{T}_1^2)\cup[\alpha]&=\big(c_1\,t^2-2c_3\,t\big)\,e^{12345678};\\[5pt]
\hbox{PD}(\mathbb{T}_2^2)\cup[\alpha]&=c_2\,t\,e^{12345678};\\[5pt]
\hbox{PD}(\mathbb{T}_3^2)\cup[\alpha]&=-\,c_3\,t\,e^{12345678}.
\end{aligned}
\right.
\end{equation}

Therefore, for every $t\in\R,t\neq 0$ we get 
$$\ker([\omega_t]^2\cup)\cap\ker(\hbox{PD}(\mathbb{T}^2_1)\cup)\cap\ker(\hbox{PD}(\mathbb{T}^2_2)\cup)\cap\ker(\hbox{PD}(\mathbb{T}^2_3)\cup)=\{0\}$$
and by formula \eqref{orbifold-formula} applied to $\omega_t$ and $\hat{\omega}_t$, for any given $t\neq 0$ and $t\neq 2$ we get 
$$
\ker([\hat{\omega}_t]^2\cup)=\{0\}.
$$
We conclude that the symplectic blowup $(\hat{M},\hat{\omega}_t)$ of $(\tilde{M},\tilde{\omega}_t)$ satisfies HLC. 
This completes the proof.
\end{proof}

It was shown in \cite[p.596]{CFM} that $\hat{M}$ has a nontrivial Massey product, and thus is not formal and admits no K\"{a}hler structures.
As a consequence we can summarize our conclusions as follows.

\begin{theorem}\label{main}
The smooth manifold $\hat{M}$ has no K\"ahler structures, but the space of symplectic forms is nonempty and admits both HLC and non-HLC structures in the same connected component. 
\end{theorem}


\begin{thebibliography}{48}

\bibitem{AG1} A. Andrada and A. Garrone, Construction of symplectic solvmanifolds satisfying the hard-Lefschetz condition, {\em Linear Algebra Appl.} {\bf 706} (2025), 70--100.
\bibitem{AG2} A. Andrada and A. Garrone, Symplectic solvmanifolds not satisfying the hard-Lefschetz condition, {\tt arXiv:2505.08113}
\bibitem{Bai56}W. L. Baily, The decomposition theorem for $V$-manifolds, {\em Amer. J. Math.} \textbf{78} (1956), no. 4, 862--888.
\bibitem{Br} J.-L. Brylinski, A differential complex for Poisson manifolds,
{\em J. Differential Geom.} {\bf 28} (1988), no.~1, 93--114.

\bibitem{Cav2} G.~R. Cavalcanti, New aspects of the $dd^c$-lemma, PhD Thesis, University of Oxford, (2004).
\bibitem{Cav1} G.~R. Cavalcanti, The Lefschetz property, formality and blowing up in symplectic geometry, {\em Trans. Amer. Math. Soc.} {\bf 359} (2007), no. 1, 333--348.
\bibitem{CFM}
G.~R. Cavalcanti, M. Fern\'{a}ndez, V. Mu\~{n}oz, Symplectic resolutions, Lefschetz property
and formality, {\em Adv. Math.} {\bf 218} (2008) 576--599.
\bibitem{cho} Y. Cho, Hard Lefschetz property of symplectic structures on Compact K\"{a}hler manifolds, {\em Trans. Amer. Math. Soc.}, \textbf{368} (2016), 8223--8248.
\bibitem{dBT1} P. de Bartolomeis and A. Tomassini, On solvable generalized Calabi-Yau manifolds, {\em Ann. Inst. Fourier} {\bf 56} (2006), 1281--1296.
\bibitem{DGMS} P. Deligne, P. Griffiths,
J. Morgan, D. Sullivan, Real Homotopy Theory of K\"ahler
Manifolds, {\em Inventiones Math.} {\bf 29} (1975) pp. 245--274.
\bibitem{FM} M. Fern\'{a}ndez, V. Mu\~{n}oz, An $8$-dimensional non-formal simply connected symplectic manifold, {\em Ann. of Math. (2)} \textbf{167} (2008), 1045--1054.
\bibitem{Joy00} D.D. Joyce, \emph{Compact manifolds with special holonomy}, Oxford Mathematical Monographs, Oxford
University Press, Oxford, 2000.
\bibitem{GH} P. Griffiths, J. Harris, {\em Principles of Algebraic Geometry}, Interscience Publ., New York (1978).
\bibitem{K1} H. Kasuya, Minimal models, formality, and hard Lefschetz properties of solvmanifolds with local systems, {\em J. Differential Geom.} \textbf{93} (2013), 269--297 https://doi.org/10.4310/jdg/1361800867
\bibitem{K2} H. Kasuya, Formality and hard Lefschetz property of aspherical manifolds, {\em Osaka J. Math.} \textbf{50} (2) (2013), 439--
455. 
\bibitem{LT} F. Lusetti, A. Tomassini, Hard Lefschetz Condition on symplectic non-K\"ahler solvmanifolds, {\em Math. Z.} \textbf{311}, 74 (2025) https://doi.org/10.1007/s00209-025-03878-5
\bibitem{Ma} Y.I. Manin, {\em Three constructions of Frobenius manifolds: a comparative study}, Survey in Diff. Geom. 2000, International Press 
Vol. 7 2000, 497--554.
\bibitem{Mathieu} O. Mathieu, Harmonic cohomology classes of symplectic manifolds, {\em Comment. Math. Helv.}, {\bf 70} (1995), 1--9.
\bibitem{McD}
D. McDuff, Examples of symplectic simply connected manifolds with no K\"ahler structure, {\em J. Differential Geom.} \textbf{20} (1984), 267--277.
\bibitem{McDS} 
D. McDuff, D. Salamon, {\em Introduction to symplectic topology}, Oxford Mathematical Monographs, Oxford University Press, 1995.
\bibitem{Merkulov} S. A. Merkulov, Formality of canonical symplectic complexes and Frobenius manifolds, {\em 
Internat. Math. Res. Notices.} {\bf 14} (1998), 727--733.
\bibitem{SA56} I. Satake, On a generalization of the notion of manifold, \emph{Proc. Nat. Acad. Sci. U.S.A.} {\bf 42} (1956), n.6, p. 359--363.
\bibitem{Yan} D. Yan, Hodge structure on symplectic manifolds,
{\em Adv. Math.} {\bf 120} (1996), no.~1, 143--154.
\end{thebibliography}
\end{document}